\documentclass[11pt,a4paper,reqno]{amsart}
\usepackage[usenames]{color}
\usepackage{enumerate}
\usepackage[all]{xy}
\input{mathrsfs.sty}
\usepackage{filecontents}
\usepackage{tikz}
\usepackage{color}
\usepackage{amsmath, amsthm, amscd, amsfonts, amssymb}

\def\MR#1{}
\newtheorem{theorem}{Theorem}[section]
\newtheorem{lemma}[theorem]{Lemma}
\newtheorem{proposition}[theorem]{Proposition}
\newtheorem{corollary}[theorem]{Corollary}
\theoremstyle{definition}
\newtheorem{definition}[theorem]{Definition}

\newtheorem{remark}{Remark}[section]
\numberwithin{equation}{section}

\begin{document}
\noindent {\footnotesize\tiny}\\[1.00in]
\textcolor[rgb]{0.00,0.00,1.00}{}
\thanks{$^*$Corresponding author}
\title[]{Tauberian Conditions for Almost Convergence in a Geodesic Metric Space}
\maketitle
\begin{center}
	{\sf Hadi Khatibzadeh$^{1}$ and Hadi Pouladi$^{*2}$}\\
	{\footnotesize{\it $^{1,2}$ Department of Mathematics, University
			of Zanjan, P. O. Box 45195-313, Zanjan, Iran.\\ Email:$^1$hkhatibzadeh@znu.ac.ir, $^{2}$hadi.pouladi@znu.ac.ir.}}
\end{center}
\begin{abstract}
In the present paper, after recalling the Karcher mean in Hadamard spaces, we study the relation between convergence, almost convergence and mean convergence (respect to the defined mean) of a sequence in Hadamard spaces. These results extend Tauberian conditions from Banach spaces to Hadamard spaces. Also, we show that every almost periodic sequence in Hadamard spaces is almost convergent.
\end{abstract}
\section{Introduction}
For a sequence $\{x_n\}$ in a linear space, the Vallee-Poussin means of the sequence for each $k$ is defined by:
$$ \frac{1}{n}\displaystyle\sum_{i=0}^{n-1} x_{k+i}, $$
(see \cite{Butzer1993}) and the Cesaro mean by:
$$ \frac{1}{n}\displaystyle\sum_{i=0}^{n-1} x_i. $$
\noindent Lorentz \cite{lorentz1948} defined the concept of {\it almost convergence} for a bounded sequence $\{x_n\}$ of real (or complex) scalars using two approaches. The first approach is using Banach limits, and the second one is based on the uniform convergence of the above Vallee-Poussin means respect to $k$. In fact, he showed that  these two approaches are equivalent. Clearly convergence of sequence $\{x_n\}$ implies almost convergence of the sequence and almost convergence implies convergence of the Cesaro mean of the sequence, i.e.,
\begin{equation*}
x_n\underset{n}{\longrightarrow} y \quad \Longrightarrow \quad \frac{1}{n}\displaystyle\sum_{i=0}^{n-1} x_{k+i} \underset{n}{\longrightarrow} y ~\text{uniformly in}~k \quad \Longrightarrow \quad \frac{1}{n}\displaystyle\sum_{i=0}^{n-1} x_i \underset{n}{\longrightarrow} y
\end{equation*}
For the reverse directions we need some sufficient conditions, which are called the Tauberian conditions and considered by Lorentz \cite{lorentz1948} for scalar sequences.  Then Kuo \cite{Kuo2009} extended the Tauberian conditions from real sequences to vector sequences in Banach spaces. In this paper, after the definition of mean and almost convergence for a sequence in a Hadamard space, we extend the Tauberian conditions in this setting.

In the next section, we present some preliminaries including the basic concepts of Hadamard spaces and the required lemmas to state and prove the main results. We also define the Karcher mean as well as the concept of the ergodic and almost convergence of a sequence respect to this mean. In Sections 3 and 4, Tauberian theorems are studied respectively for metric and weak convergence. Finally, Section 5 is devoted to prove almost convergence of almost periodic sequences in Hadamard spaces.

\section{Preliminaries}
In metric space $(X,d)$, a geodesic between two points $x,y\in X$ is a map
$$ \gamma :[0,d(x,y)]\longrightarrow X, $$
such that $\gamma(0)=x,\  \gamma\big(d(x,y)\big)=y \ \text{and} \  d\big(\gamma(t),
\gamma(t')\big)=|t-t'|, \quad \forall t,t'\in[0,1]$.
The metric space that every it's two points are joined by a geodesic, is said geodesic space and it is said uniquely geodesic if between any two points there is exactly one geodesic. The image of $\gamma$ is called a geodesic segment and denoted by $[x,y]$ for uniquely geodesic spaces, also in such spaces $m_t=(1-t)x\oplus ty$ for every $t\in[0,1]$ is the unique point in the segment $[x,y]$ such that:\\
\centerline{$d(x,m_t)=td(x,y)$ and $d(y,m_t)=(1-t)d(x,y)$.}

A geodesic triangle $\triangle:=\triangle(x_1,x_2,x_3)$ in a geodesic metric space $(X,d)$ consists of three points $x_1, x_2, x_3\in X$ as vertices and three geodesic segments joining each pair of vertices as edges. A comparison triangle for the geodesic triangle $\triangle$ is the triangle $\overline{\triangle}:=\overline{\triangle}(x_1,x_2,x_3):=\triangle(\overline{x_1},\overline{x_2},\overline{x_3})$ in the Euclidean space $\Bbb{R}^2$ such that $d(x_i,x_j)=d_{\Bbb{R}^2}(\overline{x_i},\overline{x_j})$ for all $i,j=1,2,3$. A geodesic space $X$ is said to be a $CAT(0)$ space if for each geodesic triangle $\triangle$ in $X$ and its comparison triangle $\overline{\triangle}$ in $\Bbb{R}^2$, the $CAT(0)$ inequality
\begin{equation*}
d(x,y)\leqslant d_{\Bbb{R}^2}(\overline{x},\overline{y}),
\end{equation*}
is satisfied for all $x,y\in \triangle$ and all comparison points $\overline{x}, \overline{y}\in \overline{\triangle}$. 

A $CAT(0)$ space $(X,d)$ is a geodesic metric space which satisfies the $CN-$inequality 
\begin{equation}\label{cat0ineq}
d^2(x, m)\leqslant \frac{1}{2}d^2(x,y)+\frac{1}{2}d^2(x,z)-\frac{1}{4}d^2(y,z),
\end{equation}
where $x,y,z\in X$ and $m$ is the midpoint of the segment $[y,z]$, i.e.,\linebreak $d(m,y)=d(m,z)=\frac{1}{2}d(z,y)$ \cite{bridson2011metric}. Also in \cite[Lemma 2.5]{DHOMPONGSA20082572} and \cite[page 163]{bridson2011metric}, we find out that a geodesic metric space is a $CAT(0)$ space if and only if for every three points $x_0 ,x_1, y \in X$ and for every $0<t<1$
\allowdisplaybreaks\begin{equation}
d^2(y, x_t)\leqslant(1-t)d^2(y,x_0)+td^2(y,x_1)-t(1-t)d^2(x_0,x_1),
\end{equation}
where $x_t=(1-t)x_0\oplus tx_1$ for every $t\in[0,1]$, the above inequality is known as strong convexity of the function $d^2$ respect to each argument. A $CAT(0)$ space is uniquely geodesic. A complete $CAT(0)$ space is said Hadamard space. From now, we denote every Hadamard space by $\mathscr H$. In Hadamard space any nonempty closed convex subset $S$  is Chebyshev i.e., $P_Sx=\{s\in S: d(x,S)=d(x,s) \}$ is singleton, where $d(x,S):=\displaystyle\inf_{s\in S} d(x,s)$\cite{bacak2014convex}. Thus, the metric projection on nonempty closed convex subset $S$ of a Hadamard space $\mathscr H$  is the following map:
\begin{equation*}
P:\mathscr H\longrightarrow S \quad x\mapsto P_Sx,
\end{equation*}
where $P_Sx$ is the nearest point of $S$ to $x$ for all $x\in \mathscr H$. A well-known fact implies that
\begin{equation}\label{projection}
d^2(x,P_Sx)+d^2(P_Sx,y)\leqslant d^2(x,y),\ \forall y\in S
\end{equation}
(see \cite{bacak2014convex} and also \cite{dehghan2012}).

Let $(X,d)$ be a $CAT(0)$ space, a function $f:X\longrightarrow \Bbb R$ is said to be convex if for all $x,y\in X$ and for all $\lambda\in [0,1]$
\begin{equation*}
f\big((1-\lambda )x\oplus\lambda y\big)\leqslant (1-\lambda)f(x)+\lambda f(y),
\end{equation*}
clearly the metric function $d$ on $CAT(0)$ space $X$ is convex. Also, $f$ is said to be $\gamma$-strongly convex with $\gamma>0$ if for all $x,y\in \mathscr H$
\begin{equation*}
f\big(\lambda x\oplus (1-\lambda)y\big)\leqslant \lambda f(x)+(1-\lambda)f(y)-\lambda(1-\lambda)\gamma d^2(x,y).
\end{equation*}
Clearly by definition of $CAT(0)$ space, the metric function $d^2$ on $CAT(0)$ space $X$ is $\gamma$-strongly convex respect to each argument with $\gamma=1$. A function $f: X\longrightarrow \mathbb{R}$ is said to be lower semicontinuous (shortly, lsc) if the set $\{x\in X : f(x)\leqslant \alpha\}$ is closed for all $\alpha\in \Bbb R$. Any lsc, strongly convex function in a Hadamard space has a unique minimizer \cite{bacak2014convex}.\\
The following lemma contains some inequalities that are satisfied in any Hadamard space, and we use them in the next section.

\begin{lemma}$($\cite{CHAOHA2006983,papadopoulos2005}$)$.\label{l2}
Let $(X,d)$ be a $CAT(0)$ space. Then for all $x,y,z\in X$ and\linebreak $t,s\in[0,1]$; we have:\\
$1)\quad d\big((1-t)x\oplus ty,z\big)\leqslant (1-t)d(x,z)+td(y,z)$.\\
$2)\quad d\big((1-t)x\oplus ty,(1-s)x\oplus sy\big)= |t-s|d(x,y)$.\\
$3)\quad d\big((1-t)z\oplus tx,(1-t)z\oplus ty\big)\leqslant td(x,y).$
\end{lemma}

Berg and Nikolaev in \cite{Berg2008} introduced the notion of quasilinearization that is the map $\langle \cdot,\cdot\rangle:(X\times X)\times (X\times X)\longrightarrow \Bbb R$ defined by
\begin{equation}\label{e12}
\langle\overset{\rightarrow}{ab},\overset{\rightarrow}{cd}\rangle=\frac{1}{2}\big\{d^2(a,d)+d^2(b,c)-d^2(a,c)-d^2(b,d) \big\} \quad a,b,c,d\in X,
\end{equation}
where a vector $\overset{\rightarrow}{ab}$ or $ab$ denotes a pair $(a,b)\in X\times X$. Also in \cite{Berg2008} they proved that a geodesically connected metric space is a $CAT(0)$ space if and only if it satisfies the Cauchy-Schwarz inequality as:
\begin{equation*}
\langle ab, cd\rangle\leqslant d(a,b)d(c,d) \quad (a,b,c,d\in X).
\end{equation*}

Now we define the notion of $\triangle-$convergence in $CAT(0)$ spaces that is weaker than the convergence respect to metric and it is an alternative of weak convergence in these spaces.\\
In a $CAT(0)$ space $X$, for a bounded sequence $\{x_n\}$ if for $x\in X$ we set\linebreak $r(x,\{x_n\})=\displaystyle\limsup_{n\rightarrow \infty} d(x,x_n)$, the asymptotic radius of $\{x_n\}$ is defined as follows:
$$r(\{ x_n\} )=\inf \big\{ r(x,\{ x_n\} ): x\in X \big\},$$
and the asymptotic center is the set
$$ A(\{ x_n\} )=\big\{x\in X: r(x,\{ x_n\} )= r(\{ x_n\} ) \big\}.$$
It is known that in a Hadamard space, $A(\{ x_n\} )$ is singleton\cite{KIRK20083689}. The notion of\linebreak $\triangle-$convergence first introduced by Lim\cite{lim1976} as follows.
\begin{definition}\label{weakcondef1}
A sequence $\{ x_n\}$ is said $\triangle-$convergent to $x$ if $x$ is the unique asymptotic center of $\{ x_{n_j}\}$ for every subsequence $\{ x_{n_j}\}$ of $\{ x_n\}$. The point $x$ is said $\triangle-\lim$ of $\{ x_n\}$ and denoted as $\bigtriangleup-\displaystyle\lim_n x_n=x$ or $x_n\overset{\bigtriangleup}{\longrightarrow}x$.
\end{definition}
\begin{lemma}$($see \cite{KIRK20083689}$)$.\label{lconsubseq}
Every bounded sequence in $CAT(0)$ space has a $\bigtriangleup -convergent$ subsequence. Also every closed convex subset of a Hadamard space is $\bigtriangleup-$closed in the sense that it contains all $\bigtriangleup-\lim$ points of every $\bigtriangleup$-convergent subsequence. 
\end{lemma}
We have two other equivalent definitions for the notion of  $\triangle-$convergence by the next two propositions.
\begin{proposition}\label{weakcondef2}$($\cite[Proposition 5.2]{espinola2009}$)$
A sequence $\{x_n\}$ in a Hadamard space $(\mathscr H,d)$, $\triangle-$converges to $x$ if and only if $\displaystyle\lim_{n\rightarrow\infty}d(x,P_Ix_n)=0$ for all geodesics $I$ issuing from the point $x$, where $P_I:\mathscr H\longrightarrow I$ is the projection map.
\end{proposition}
\begin{proposition}\label{weakcondef3}$($\cite[Theorem 2.6]{kakavandi2013}$)$
Let $(X,d)$ be a $CAT(0)$ space, $\{x_n\}$ be a sequence in $X$ and $x\in X$. Then $\{x_n\}$ $\triangle-$converges to $x$ if and only if
\begin{equation*}
\displaystyle\limsup_{n\rightarrow\infty} \langle xx_n,xy \rangle\leqslant 0, \quad \forall y\in X.
\end{equation*}
\end{proposition}

\begin{definition}\label{Karcher mean}
Given a finite number of points $x_0,\ldots ,x_{n-1}$ in a Hadamard space, we define the functions
\begin{equation}\label{e22}
\mathcal{F}_n(x)=\frac{1}{n}\displaystyle\sum_{i=0}^{n-1} d^2(x_i,x),
\end{equation}
and
\begin{equation}\label{e23}
\mathcal{F}_n^k(x)=\frac{1}{n}\displaystyle\sum_{i=0}^{n-1} d^2(x_{k+i},x).
\end{equation}
From \cite[p.41 Proposition 2.2.17]{bacak2014convex} we know that these functions have unique minimizers. For $\mathcal{F}_n(x)$ (resp. $\mathcal{F}_n^k(x)$) the unique minimizer is denoted by $\sigma_n(x_0,\ldots ,x_{n-1})$ or shortly, $\sigma_n$,  (resp. $\sigma_n^k(x_k,\ldots ,x_{k+n-1})$ or shortly, $\sigma_n^k$) and it is called the mean of $x_0,\ldots ,x_{n-1}$ (resp. $x_k,\ldots ,x_{k+n-1}$). These mean is known as the Karcher mean of $x_0,\ldots ,x_{n-1}$ (resp. $x_k,\ldots ,x_{k+n-1}$) (see \cite{karcher1977}).
\end{definition}
\begin{remark} \label{Karcher mean Banach spaces}
The Karcher mean is also defined in each reflexive and strictly convex Banach spaces if we replace $d(\cdot,\cdot)$ by $\|\cdot\|$ in Definition \ref{Karcher mean}. Reflexivity ensures the existence of the minimizer and the strict convexity ensures uniqueness of the minimizer. In spite of Hilbert spaces which in the Karcher mean is the same arithmetic (linear) mean, in general reflexive and strictly convex Banach spaces they are different. Because matching of these means is equivalent to the parallelogram identity, which is a characterization of Hilbert spaces\cite{danamir2013}.
\end{remark}

\section{Tauberian Conditions for Metric Convergence}
A sequence $\{x_n\}$ in a Hadamard space $\mathscr H$ is called the Cesaro convergent or the mean convergent  (resp. almost convergent) to $x\in X$, if $\sigma_n$ (resp. $\sigma_n^k$) converges (resp. converges uniformly in $k$) to $x$. In this section, we present some Tauberian theorems for these means.
We need the next lemma to prove the Tauberian theorems for the Karcher mean.
\begin{lemma}\label{l15}
Let $\{x_n\}$ be a sequence in Hadamard space $\mathscr H$. Then for $\sigma_n^k$ defined as the above, for each $y\in\mathscr H$ and $k\geqslant 1$, we have:
\begin{enumerate}
\renewcommand{\theenumi}{\roman{enumi}}
\item\label{l15i1}
$\displaystyle d^2\big( \sigma_n^k ,y\big)\leqslant \frac{1}{n}\displaystyle\sum_{i=0}^{n-1} d^2(x_{k+i},y)-\frac{1}{n}\displaystyle\sum_{i=0}^{n-1} d^2(x_{k+i},\sigma_n^k).$
\item\label{l15i2} $\displaystyle d\big( \sigma_n^k ,y\big)\leqslant \frac{1}{n}\displaystyle\sum_{i=0}^{n-1} d(x_{k+i},y).$
\end{enumerate}
\end{lemma}
\begin{proof}
\eqref{l15i1}.
Since $\sigma_n^k$ is the unique minimizer of $\mathcal{F}_n^k(x)$ defined in \eqref{e23} and by the strong convexity of this function, for $0<\lambda <1$ we have:
\allowdisplaybreaks\begin{eqnarray}
\mathcal{F}_n^k(\sigma_n^k)&\leqslant & \mathcal{F}_n^k\big(\lambda\sigma_n^k\oplus (1-\lambda)y\big) \nonumber \\
&\leqslant &\lambda\mathcal{F}_n^k(\sigma_n^k)+ (1-\lambda)\mathcal{F}_n^k(y)- \lambda(1-\lambda)d^2\big( \sigma_n^k ,y\big). \nonumber
\end{eqnarray}
Therefor we obtain
\begin{equation*}
\lambda d^2\big( \sigma_n^k ,y\big)\leqslant \mathcal{F}_n^k(y)-\mathcal{F}_n^k(\sigma_n^k).
\end{equation*}
Letting $\lambda\rightarrow 1$ implies:
\begin{eqnarray}\label{ee1}
d^2\big( \sigma_n^k ,y\big)&\leqslant &\mathcal{F}_n^k(y)-\mathcal{F}_n^k(\sigma_n^k) \nonumber \\
&= &\frac{1}{n}\displaystyle\sum_{i=0}^{n-1} d^2(x_{k+i},y)-\frac{1}{n}\displaystyle\sum_{i=0}^{n-1} d^2(x_{k+i},\sigma_n^k),
\end{eqnarray}
which is the intended result. In particular, we have
\begin{equation*}
d^2\big( \sigma_n^k ,y\big)\leqslant \frac{1}{n}\displaystyle\sum_{i=0}^{n-1} d^2(x_{k+i},y).
\end{equation*}
\eqref{l15i2}. Triangle inequality yields:
\begin{equation*}
d^2\big( \sigma_n^k ,y\big)+d^2\big( y,x_{k+i}\big)-2d\big( \sigma_n^k ,y\big)d\big( y,x_{k+i}\big)\leqslant d^2\big( \sigma_n^k ,x_{k+i}\big),
\end{equation*}
hence
\begin{equation*}
d^2\big( y,x_{k+i}\big)\leqslant d^2\big( \sigma_n^k ,x_{k+i}\big)+2d\big( \sigma_n^k ,y\big)d\big( y,x_{k+i}\big)-d^2\big( \sigma_n^k ,y\big).
\end{equation*}
So summing up over $i$ from $0$ to $n-1$ and multiplying by $\frac{1}{n}$ imply:
\begin{equation}\label{convkarcher1}
\frac{1}{n}\displaystyle\sum_{i=0}^{n-1}d^2\big( y,x_{k+i}\big)\leqslant \frac{1}{n}\displaystyle\sum_{i=0}^{n-1}d^2\big( \sigma_n^k ,x_{k+i}\big)+2d\big( \sigma_n^k ,y\big)\frac{1}{n}\displaystyle\sum_{i=0}^{n-1}d\big( y,x_{k+i}\big)-d^2\big( \sigma_n^k ,y\big).
\end{equation}
On the other hand, by \eqref{ee1} we have:
\begin{equation}\label{convkarcher2}
\frac{1}{n}\displaystyle\sum_{i=0}^{n-1} d^2(x_{k+i},\sigma_n^k)\leqslant \frac{1}{n}\displaystyle\sum_{i=0}^{n-1} d^2(x_{k+i},y)-d^2(\sigma_n^k ,y).
\end{equation}
\eqref{convkarcher1} and \eqref{convkarcher2} show that
\begin{equation*}
d\big( \sigma_n^k ,y\big)\leqslant \frac{1}{n}\displaystyle\sum_{i=0}^{n-1} d(x_{k+i},y).
\end{equation*}
\end{proof}
In the next theorem, we consider the relation between convergence and the almost convergence.
\begin{theorem}\label{tautheokarcher1}
Let $\{x_n\}$ be a sequence in Hadamard space $\mathscr H$. Then $\{x_n\}$ converges to $y$ if and only if $\sigma_n^k$ defined as the unique minimizer of \eqref{e23} converges to $y$ uniformly in $k\geqslant 0$ $($or the sequence $\{x_n\}$ almost converges to $y)$ and $\{x_n\}$ is asymptotically regular $($i.e., $d(x_n,x_{n+1})\to0$ as $n\to\infty)$.
\end{theorem}
\begin{proof}
\textbf{Necessity.} By Part \eqref{l15i1} of Lemma \ref{l15}, we have:
\begin{equation}\label{e25}
d^2\big( \sigma_n^k ,y\big)\leqslant \frac{1}{n}\displaystyle\sum_{i=0}^{n-1} d^2(x_{k+i},y).
\end{equation}
Since the sequence $\{x_n\}$ converges strongly to $y$, $d^2(x_n,y)\longrightarrow 0$ and hence the right side of \eqref{e25} converges to zero uniformly in $k$, consequently $\sigma_n^k$ converges to $y$ uniformly in $k$.  Also it is clear that the sequence $\{x_n\}$ is asymptotically regular.\\
\textbf{Sufficiency.} Let $\sigma_n^k$ converge to $y$ uniformly in $k\geqslant 0$ and $\{x_n\}$ is asymptotically regular. By $CN-$inequality, we have:
\allowdisplaybreaks\begin{eqnarray}
0&\leqslant &d^2\big( \sigma_n^k,\frac{1}{2}x_k\oplus\frac{1}{2}y\big)\nonumber \\
&\leqslant &\frac{1}{2}d^2(\sigma_n^k,x_k)+\frac{1}{2}d^2(\sigma_n^k,y)-\frac{1}{4}d^2(x_k,y).\nonumber
\end{eqnarray}
Therefore by Part \eqref{l15i1} of Lemma \ref{l15}, we obtain:
\allowdisplaybreaks\begin{eqnarray}
d^2(x_k,y)&\leqslant & 2d^2(\sigma_n^k,x_k)+2d^2(\sigma_n^k,y) \nonumber \\
&\leqslant &\frac{2}{n}\displaystyle\sum_{i=0}^{n-1}d^2(x_{k+i},x_k)+2d^2(\sigma_n^k,y)\nonumber \\
&= &\frac{2}{n}\bigg\{d^2(x_k,x_{k+1})+d^2(x_k,x_{k+2})+\cdots +d^2(x_k,x_{k+n-1}) \bigg\}+2d^2(\sigma_n^k,y)\nonumber \\
&\leqslant &\frac{2}{n}\bigg\{d^2(x_k,x_{k+1})+\bigg(\displaystyle\sum_{i=k}^{k+1}d(x_i,x_{i+1})\bigg)^2+\cdots +\bigg(\displaystyle\sum_{i=k}^{k+n-2}d(x_i,x_{i+1})\bigg)^2 \bigg\} \nonumber\\
& & + 2d^2(\sigma_n^k,y)\nonumber\\
&\leqslant &\frac{2}{n}\bigg(\displaystyle\sup_{i\geqslant k} d(x_i,x_{i+1})\bigg)^2\big( 1+2^2+\cdots (n-1)^2\big)+2d^2(\sigma_n^k,y) \nonumber\\
&= &\frac{2}{n}\frac{(n-1)n(2n-1)}{2}\bigg(\displaystyle\sup_{i\geqslant k} d(x_i,x_{i+1})\bigg)^2+2d^2(\sigma_n^k,y)  \nonumber\\
&= &(n-1)(2n-1)\bigg(\displaystyle\sup_{i\geqslant k} d(x_i,x_{i+1})\bigg)^2+2d^2(\sigma_n^k,y).\nonumber
\end{eqnarray}
From asymptotic regularity of $\{x_n\}$, $d(x_{n+1},x_n)\longrightarrow 0$. Taking $\limsup$ when $k\to\infty$, we get:
\begin{equation*}
\displaystyle\limsup_{k\rightarrow\infty} d^2(x_k,y)\leqslant 2\displaystyle\limsup_{k\rightarrow\infty} d^2(\sigma_n^k,y)\leqslant 2\displaystyle\sup_{k\geqslant 1} d^2(\sigma_n^k,y).
\end{equation*}
Since $\sigma_n^k$ is uniformly convergent to $y$, letting $n\longrightarrow \infty$ completes the proof.
\end{proof}
For the relation between the mean convergence and the almost convergence defined above, we present the following Tauberian condition:
\begin{equation}\label{taucondkarchermean2}
\displaystyle\lim_{n\rightarrow\infty}\displaystyle\sup_{k\geqslant 0}\bigg(\frac{1}{n+k}\displaystyle\sum_{i=0}^{k-1} \Big(d^2(x_i,\sigma_n^k)-d^2(x_i,\sigma_k)\Big)\bigg)=0.
\end{equation}
\begin{theorem}\label{tautheokarcher2}
For the sequence $\{x_n\}$ in Hadamard space $\mathscr H$, $\sigma_n^k$ defined as the unique minimizer of \eqref{e23} converges to $y$ uniformly in $k\geqslant 0($or the sequence $\{x_n\}$ is almost convergent to $y)$ if and only if $\sigma_n$ defined as the unique minimizer of \eqref{e22} converges to $y$ $($or the sequence $\{x_n\}$ is mean convergent to $y)$ and \eqref{taucondkarchermean2} is satisfied.
\end{theorem}
\begin{proof}
\textbf{Necessity.} With getting $k=0$, it is obvious.\\
\textbf{Sufficiency.} Let $\sigma_n$ converge to $y$ and \eqref{taucondkarchermean2} is satisfied. First by $CN-$inequality we have:
\allowdisplaybreaks\begin{eqnarray}
0&\leqslant &d^2\big(\sigma_{n+k},\frac{1}{2}\sigma_n^k\oplus\frac{1}{2}y \big)\nonumber \\
&\leqslant &\frac{1}{2}d^2(\sigma_{n+k},\sigma_n^k)+\frac{1}{2}d^2(\sigma_{n+k},y)-\frac{1}{4}d^2(\sigma_n^k,y).\nonumber
\end{eqnarray}
Therefore by Part \eqref{l15i1} of Lemma \ref{l15}, and definition of $\sigma_n^k$ in inequality $\star$ and also definition of $\sigma_k$ in inequality $\star\star$, we obtain:
\allowdisplaybreaks\begin{eqnarray}
d^2(\sigma_n^k,y)&\leqslant & 2d^2(\sigma_{n+k},\sigma_n^k)+2d^2(\sigma_{n+k},y) \nonumber \\
&\leqslant &2\bigg(\frac{1}{n+k}\displaystyle\sum_{i=0}^{n+k-1}d^2(x_i,\sigma_n^k)- \frac{1}{n+k}\displaystyle\sum_{i=0}^{n+k-1}d^2(x_i,\sigma_{n+k})\bigg)+2d^2(\sigma_{n+k},y) \nonumber \\
&\leqslant &2\bigg(\frac{1}{n+k}\displaystyle\sum_{i=0}^{k-1}d^2(x_i,\sigma_n^k)+ \frac{1}{n+k}\displaystyle\sum_{i=k}^{n+k-1}d^2(x_i,\sigma_n^k)    \nonumber \\
& &- \frac{1}{n+k}\displaystyle\sum_{i=0}^{n+k-1}d^2(x_i,\sigma_{n+k})\bigg)+2d^2(\sigma_{n+k},y) \nonumber \\
&\overset{\star}{\leqslant} &2\bigg(\frac{1}{n+k}\displaystyle\sum_{i=0}^{k-1}d^2(x_i,\sigma_n^k)+ \frac{1}{n+k}\displaystyle\sum_{i=k}^{n+k-1}d^2(x_i,\sigma_{n+k})    \nonumber \\
& &- \frac{1}{n+k}\displaystyle\sum_{i=0}^{n+k-1}d^2(x_i,\sigma_{n+k})\bigg)+2d^2(\sigma_{n+k},y) \nonumber \\
&\overset{\star\star}{\leqslant} &2\bigg(\frac{1}{n+k}\displaystyle\sum_{i=0}^{k-1}\Big(d^2(x_i,\sigma_n^k)-d^2(x_i,\sigma_k) \Big)\bigg)+2d^2(\sigma_{n+k},y). \nonumber
\end{eqnarray}
Thus we get:
\begin{eqnarray}
\displaystyle\sup_{k\geqslant 0}d^2(\sigma_n^k,y)&\leqslant &2\displaystyle\sup_{k\geqslant 0}\bigg(\frac{1}{n+k}\displaystyle\sum_{i=0}^{k-1}\Big(d^2(x_i,\sigma_n^k)-d^2(x_i,\sigma_k) \Big)\bigg)+2\displaystyle\sup_{k\geqslant n}d^2(\sigma_{k},y). \nonumber
\end{eqnarray}
Letting $n\to\infty$, the proof is now complete by the assumptions.
\end{proof}
In the next theorem, we show another Tauberian condition for the relation between the mean convergence and convergence of the sequence in Hadamard spaces.
We first state an elementary lemma without proof.
\begin{lemma}\label{l7}
For a real sequence $\{a_n\}$, we have:
\begin{equation*}
\frac{1}{n+1}\sum_{k=0}^{n} a_k=a_n-\frac{1}{n+1}\sum_{k=1}^{n}k(a_k-a_{k-1}).
\end{equation*}
\end{lemma}
By Theorems \ref{tautheokarcher1} and \ref{tautheokarcher2}, we know that convergence of the sequence $\{x_n\}$ implies its mean convergence, for the reverse direction we have the following theorem.
\begin{theorem}\label{tautheokarcher3}
Let $\{x_n\}$ be a sequence in Hadamard space $\mathscr H$, and $\sigma_n$ be the mean sequence defined as the unique minimizer of \eqref{e22}. If $\sigma_n$ converges to $y$ and $nd(x_n,x_{n-1})\longrightarrow 0$ as $n\to\infty$, then $x_n$ converges to $y$.
\end{theorem}
\begin{proof}
For a fixed integer $n>0$, let $a_i=d(x_i,x_n)$ for $i=1,2,\cdots, n$, then by Part \eqref{l15i2} of Lemma \ref{l15} and Lemma \ref{l7}, we have:
\allowdisplaybreaks\begin{eqnarray}
d(x_n,y)&\leqslant &d(x_n,\sigma_{n+1})+d(\sigma_{n+1},y)\nonumber \\
&\leqslant &\frac{1}{n+1}\sum_{i=0}^{n}d(x_n,x_i)+d(\sigma_{n+1},y) \nonumber \\
&\leqslant &\frac{1}{n+1}\sum_{i=1}^{n}i\big(d(x_{i-1},x_n)-d(x_i,x_n)\big)+d(\sigma_{n+1},y) \nonumber\\
&\leqslant &\frac{1}{n+1}\sum_{i=1}^{n}id(x_{i-1},x_i)+d(\sigma_{n+1},y)\to 0,\ \ \text{as}\ n\to\infty \nonumber
\end{eqnarray}
which concludes the result.
\end{proof}
Summary of the results of Tauberian theorems for the Karcher mean is showed in the below figure:
\begin{figure}[!h]
\centerline{\includegraphics[height=5cm]{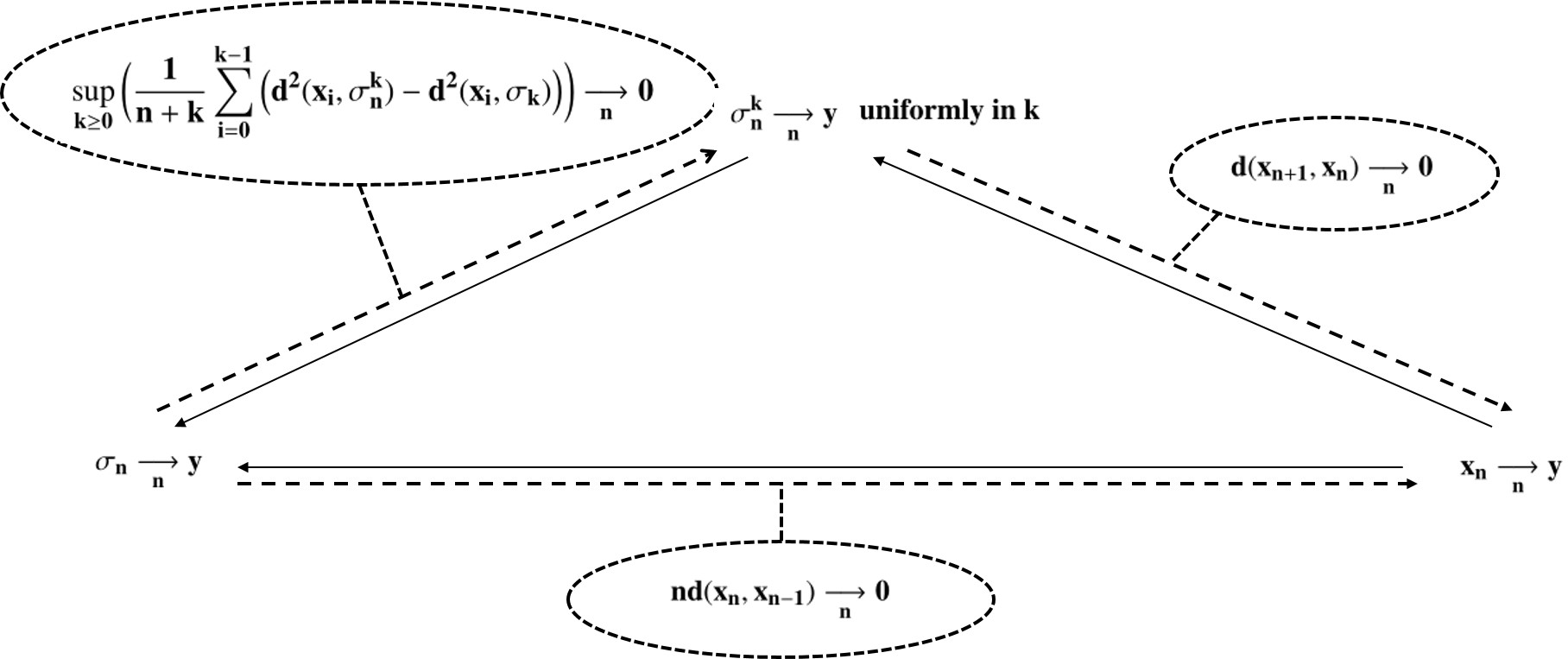}}
\end{figure}

The results of this section and Remark \ref{Karcher mean Banach spaces} propose this question: Does the Tauberian conditions hold for the Karcher mean in more general geodesic spaces for example uniformly convex geodesic metric spaces or in uniformly convex Banach spaces?
\begin{remark}
Let $(\mathscr H,d)$ be a Hadamard space. For a curve $c:[0,\infty)\longrightarrow\mathscr H$, $\sigma_T$ and $\sigma_T^s$(the Cesaro mean and the Vallee-Poussin means) are defined respectively as the unique minimizers of the functions 
\begin{equation*}
\mathcal{G}_T(y)=\frac{1}{T}\displaystyle\int_{0}^{T} d^2(c(t),y)dt,
\end{equation*}
and
\begin{equation*}
\mathcal{G}_T^s(y)=\frac{1}{T}\displaystyle\int_{0}^{T} d^2(c(t+s),y)dt.
\end{equation*}
A curve $c:[0,\infty)\longrightarrow\mathscr H$ is called the Cesaro convergent or the mean convergent  (resp. almost convergent) to $x\in X$, if $\sigma_T$ (resp. $\sigma_T^s$) converges (resp. converges uniformly in $s$) to $x$ as $T\longrightarrow\infty$. It is easy to check that the results of Theorems \ref{tautheokarcher1} and \ref{tautheokarcher2} remain true for curves with similar proofs, but Theorem \ref{tautheokarcher3} for curves remains an open question.
\end{remark}
\section{Tauberian Conditions for $\triangle-$Convergence}
In this section, we are going to show that the obtained results in the previous section hold for $\triangle-$convergence. Proofs of Corollaries 4.1, 4.2 and 4.3 is directly concluded from the proofs of Theorems 3.2, 3.3 and 3.5. In the next corollary, we see that the "if" part of Theorem \ref{tautheokarcher1} holds for $\triangle-$convergence. But this fact that $\triangle-$convergence of $\{x_n\}$ implies $\triangle-$almost convergence of $\{x_n\}$ in general Hadamard spaces is an open question for us.
\begin{corollary}\label{tautheokarcher1weak}
Let $\{x_n\}$ be a sequence in Hadamard space $\mathscr H$. If $\sigma_n^k$ defined as unique minimizer of \eqref{e23} $\triangle-$converges to $y$ uniformly in $k\geqslant 0$ $($or the sequence $\{x_n\}$\linebreak $\triangle-$almost converges to $y)$ and $\{x_n\}$ is asymptotically regular $($i.e., $d(x_n,x_{n+1})\to0$ as $n\to\infty)$, then $\{x_n\}$ $\triangle-$converges to $y$.
\end{corollary}
\begin{proof}
By nonexpansiveness of the projection mapping for all geodesic $I$ issuing from $x$, we have:
\allowdisplaybreaks\begin{eqnarray}
d^2(P_Ix_k,x) &\leqslant &2d^2(P_Ix_k, P_I\sigma_n^k)+2d^2(P_I\sigma_n^k, x) \nonumber\\
&\leqslant &2d^2(x_k, \sigma_n^k)+2d^2(P_I\sigma_n^k, x). \nonumber
\end{eqnarray}
Now the proof of Theorem \ref{tautheokarcher1} can be repeated. The result is concluded by Proposition \ref{weakcondef2}.
\end{proof}
Also for the relation between $\triangle-$mean convergence and $\triangle-$almost convergence of a sequence, we have the following theorem.
\begin{corollary}\label{tautheokarcher2weak}
For the sequence $\{x_n\}$ in Hadamard space $\mathscr H$, if $\sigma_n^k$ defined as the unique minimizer of \eqref{e23} $\triangle-$converges to $y$ uniformly in $k\geqslant 0($or the sequence $\{x_n\}$, $\triangle-$almost converges to $y)$, then $\sigma_n$ defined as the unique minimizer of \eqref{e22}\linebreak $\triangle-$converges to $y$ $($or the sequence $\{x_n\}$, $\triangle-$mean converges to $y)$. Also, $\triangle-$ convergence of $\sigma_n$ to $y$ and \eqref{taucondkarchermean2} imply $\triangle-$ convergence of $\sigma_n^k$ to $y$ uniformly in $k\geqslant 0$.
\end{corollary}
\begin{proof}
It is obvious that $\triangle-$almost convergence of the sequence $\{x_n\}$ to $x$ implies $\triangle-$mean convergence of the sequence $\{x_n\}$ to $x$, also via the condition \eqref{taucondkarchermean2} the reverse direction is true by the same proof of Theorem \ref{tautheokarcher2}. Because by nonexpansiveness of the projection mapping for all geodesic $I$ issuing from $x$, we have:
\allowdisplaybreaks\begin{eqnarray}
d^2(P_I\sigma_n^k ,x) &\leqslant &2d^2(P_I\sigma_n^k, P_I\sigma_{n+k})+2d^2(P_I\sigma_{n+k}, x) \nonumber\\
&\leqslant &2d^2(\sigma_n^k, \sigma_{n+k})+2d^2(P_I\sigma_{n+k}, x), \nonumber
\end{eqnarray}
and this completes the proof by Proposition \ref{weakcondef2}.
\end{proof}
We can also obtain suitable Tauberian condition for $\triangle-$mean convergence to\linebreak $\triangle-$ convergence of a sequence in the next corollary.
\begin{corollary}\label{tautheokarcher3weak}
Let $\{x_n\}$ be a sequence in Hadamard space $\mathscr H$, and $\sigma_n$ is the mean sequence defined as the unique minimizer of \eqref{e22}. If $\sigma_n$ $\triangle-$converges to $y$ and\linebreak $nd(x_n,x_{n-1})\longrightarrow 0$ as $n\to\infty$, then $x_n$ $\triangle-$converges to $y$.
\end{corollary}
\begin{proof}
By nonexpansiveness of the projection mapping on all geodesic $I$ issuing from $x$, we have:
\allowdisplaybreaks\begin{eqnarray}
d(P_Ix_n ,x) &\leqslant &d(P_Ix_n, P_I\sigma_{n+1})+d(P_I\sigma_{n+1}, x) \nonumber\\
&\leqslant &d(x_n, \sigma_{n+1})+d(P_I\sigma_{n+1}, x). \nonumber
\end{eqnarray} The rest of the proof is similar to the proof of Theorem \ref{tautheokarcher3} and it is concluded from Proposition \ref{weakcondef2}.
\end{proof}
As we stated in the first of this section, we don't know whether $\triangle-$convergence of $\{x_n\}$ implies $\triangle-$almost convergence of $\{x_n\}$ or not. But for $\triangle-$convergence to $\triangle-$mean convergence, we prove the next theorem in some special Hadamard spaces. The following condition as a geometric condition for nonpositive curvature metric spaces has been introduced by Kirk and Panyanak \cite{KIRK20083689} as:

$(\mathbf{Q_4})$: for points $x,y,p,q\in\mathscr H$ and any point $m$ in the segment $[x,y]$
\begin{equation*}
d(p,x)<d(x,q)\ \&\ d(p,y)<d(y,q) \Longrightarrow d(p,m)\leqslant d(m,q).
\end{equation*}
and $(\mathbf{\overline{Q_4}})$ condition as the modification of it introduced by Kakavandi \cite{kakavandi2013} as:

$(\mathbf{\overline{Q_4}})$: for points $x,y,p,q\in\mathscr H$ and any point $m$ in the segment $[x,y]$
\begin{equation*}
d(p,x)\leqslant d(x,q)\ \&\ d(p,y)\leqslant d(y,q) \Longrightarrow d(p,m)\leqslant d(m,q).
\end{equation*}
In fact if for $x,y\in \mathscr H$ set $F(x,y):=\{z\in \mathscr H\ : \ d(x,z)\leqslant d(y,z) \}$, the $(\mathbf{\overline{Q_4}})$ condition is equivalent to $F(x,y)$ is convex for any $x,y\in \mathscr H$. Hilbert spaces, $\Bbb{R}$-trees and any $CAT(0)$ space of constant curvature satisfy $(\mathbf{\overline{Q_4}})$ condition \cite{kakavandi2013,espinola2009}. We need the next lemma before stating the main result.
\begin{lemma}\label{l16}
Let $\mathscr H$ be a Hadamard space and $\{x_n\}$ be a sequence in $\mathscr H$. Let $\{\sigma_n\}$ and $\{\sigma_n^k\}$ be the means defined in Definition \ref{Karcher mean}, for each $k\geqslant 0$, we have:
\begin{enumerate}
\renewcommand{\theenumi}{\roman{enumi}}
\item\label{l16i1} $\sigma_n^k\in \overline{co}\{x_n\}$.
\item\label{l16i2} If $\{x_n\}$ is bounded, then $d\big(\sigma_n,\sigma_n^k\big)\longrightarrow 0$ as $n\rightarrow\infty$.
\end{enumerate}
\end{lemma}
\begin{proof}
\eqref{l16i1}.
Let $P:\mathscr H\longrightarrow \overline{co}\{x_n\}$ be the projection map. On the one hand, by the inequality \eqref{projection} for any $i$, we have:
\begin{equation*}
d^2(x_i,\sigma_n^k)\geqslant d^2(x_i,P\sigma_n^k)+ d^2(P\sigma_n^k, \sigma_n^k).
\end{equation*}
On the other hand, by definition of $\sigma_n^k$, we have:
\begin{equation*}
\displaystyle\frac{1}{n}\sum_{i=0}^{n-1} d^2(x_{k+i}, \sigma_n^k) \leqslant \displaystyle\frac{1}{n}\sum_{i=0}^{n-1} d^2(x_{k+i}, P\sigma_n^k).
\end{equation*}
Hence we obtain $d^2(P\sigma_n^k, \sigma_n^k)=0$, which is the requested result.\\
\eqref{l16i2}.
By Lemma \ref{l15}, Part \eqref{l15i1}, and definition of $\sigma_n$ in inequality $\star$, we get:
\allowdisplaybreaks\begin{eqnarray}
d^2\big(\sigma_n,\sigma_n^k\big)&\leqslant &\frac{1}{n}\displaystyle\sum_{i=0}^{n-1}d^2\big(x_{k+i},\sigma_n\big)- \frac{1}{n}\displaystyle\sum_{i=0}^{n-1}d^2\big(x_{k+i},\sigma_n^k\big) \nonumber \\
&\leqslant &\frac{1}{n}\displaystyle\sum_{i=0}^{n-1}d^2\big(x_i,\sigma_n\big)- \frac{1}{n}\displaystyle\sum_{i=0}^{n-1}d^2\big(x_{k+i},\sigma_n^k\big) \nonumber \\
&&+ \frac{1}{n}\displaystyle\sum_{i=0}^{k-1} d^2\big(x_{n+i},\sigma_n\big) \nonumber \\
&\overset{\star}{\leqslant} &\frac{1}{n}\displaystyle\sum_{i=0}^{n-1}d^2\big(x_i,\sigma_n^k\big)- \frac{1}{n}\displaystyle\sum_{i=0}^{n-1}d^2\big(x_{k+i},\sigma_n^k\big) \nonumber \\
&&+\frac{1}{n}\displaystyle\sum_{i=0}^{k-1} d^2\big(x_{n+i},\sigma_n\big) \nonumber \\
&\leqslant & \frac{1}{n}\displaystyle\sum_{i=0}^{n-1}d^2\big(x_{k+i},\sigma_n^k\big)- \frac{1}{n}\displaystyle\sum_{i=0}^{n-1}d^2\big(x_{k+i},\sigma_n^k \big) \nonumber\\
&& +\frac{1}{n}\displaystyle\sum_{i=0}^{k-1} d^2\big(x_i,\sigma_n^k\big)+\frac{1}{n}\displaystyle\sum_{i=0}^{k-1} d^2\big(x_{n+i},\sigma_n\big). \nonumber
\end{eqnarray}
Now boundedness of the sequence $\{x_n\}$ (and hence boundedness of $\{\sigma_n\}$ and $\{\sigma_n^k\}$) completes the proof.
\end{proof}
\begin{theorem}\label{tautheokarcher4}
Let $\{x_n\}$ be a sequence in Hadamard space $\mathscr H$ that satisfy $(\overline{Q_4})$ condition, and $\sigma_n$ be the mean sequence defined as the unique minimizer of \eqref{e22}. If $\{x_n\}$ $\triangle-$converges to $y$, then $\sigma_n$ $\triangle-$converges to $y($or $\{x_n\}$ $\triangle-$mean converges to $y)$.
\end{theorem}
\begin{proof}
First note that by $\triangle-$convergence of $\{x_n\}$, this sequence and hence $\{\sigma_n\}$ and $\{\sigma_n^k\}$ for each $k\geqslant 1$ are bounded, therefore by Part \ref{l16i2} of Lemma \ref{l16} for each $k\geqslant 1$, $d\big(\sigma_n,\sigma_n^k\big)\longrightarrow 0$ as $n\rightarrow\infty$. Also, boundedness of $\{\sigma_n\}$ implies that there exists a subsequence $\{\sigma_{n_i}\}$ of $\{\sigma_n\}$ such that $\{\sigma_{n_i}\}$ $\bigtriangleup -converges$ to $v\in \mathscr H$. Since $d\big(\sigma_n,\sigma_n^k\big)\longrightarrow 0$ for each $k\geqslant 1$, we have $\{\sigma_{n_i}^k\}$ $\bigtriangleup -converges$ to $v$ for each $k\geqslant 1$. If we show that $v=y$, then the proof is complete. Suppose to the contrary, i.e., there is a $\delta>0$ such that
\begin{equation*}
d(v,y)=\delta.
\end{equation*}
On the other hand, by $\bigtriangleup -convergence$ of $x_n$ to $y$, using Proposition \ref{weakcondef3}, we have:
\begin{equation*}
\displaystyle\limsup_{n\rightarrow\infty} \Big(d^2(v,y)+d^2(x_n,y)-d^2(x_n,v)\Big)\leqslant 0.
\end{equation*}
Hence there exists $N$ such that for any $n\geqslant N$
\begin{equation*}
d^2(x_n,y)-d^2(x_n,v)\leqslant 0.
\end{equation*}
Now we know that by $(\overline{Q_4})$ condition, the set $F(y,v)$ is convex. By Part \ref{l16i1} of Lemma \ref{l16}, $\sigma_n^N\in F(y,v)$. Also by continuity of metric function, $F(y,v)$ is closed and hence by Lemma \ref{lconsubseq}, it is $\bigtriangleup-$closed. This facts since $\sigma_{n_i}^N$ $\bigtriangleup -converges$ to $v$, replacing $n$ with $n_i$ and by $\bigtriangleup$-closedness of $F(y,v)$, imply that $d^2(y,v)=0$ i.e., $y=v$, which is a contradiction. This completes the proof.
\end{proof}
\begin{remark}
\noindent It is easy to check that Corollaries \ref{tautheokarcher1weak}, \ref{tautheokarcher2weak} and Theorem \ref{tautheokarcher4} for $\bigtriangleup -convergence$ remains true with the analogous arguments for curves in Hadamard spaces.
\end{remark}
\section{Almost Convergence of Almost Periodic Sequences}
In \cite{bohr1947almost,lorentz1948} we see that every almost periodic real sequence is almost convergent. Now we prove it in Hadamard spaces.  
\begin{definition}[Periodic and Almost Periodic Sequences]
Let $\{x_n\}$ be a sequence in metric space $(X,d)$, we call this sequence is periodic with the period $p$ (or $p-$periodic) if there exists a positive integer $p$ such that $x_{n+p}=x_n$ for all $n$.\\
A sequence $\{x_n\}$ is called almost periodic if for each $\epsilon>0$ there are natural numbers $L=L(\epsilon)$ and $N=N(\epsilon)$ such that any interval $(k,k+L)$ where $k\geqslant 0$ contains at least one integer $p$ satisfying
\begin{equation}
d(x_{n+p},x_n)<\epsilon, \quad\quad \forall n\geqslant N.
\end{equation}
\end{definition}
We need the next lemma for proving almost convergence of almost periodic sequences in Hadamard spaces.
\begin{lemma}\label{lemstconfun2}$($see \cite[Lemma 4.3]{khatibzadehpouladi2019}$)$
Let $(\mathscr H, d)$ be a Hadamard space and $\{f_n^k\}_{k,n}$ be a sequence of convex functions on $\mathscr H$. If $\{x_n^k\}_{k,n}$ is a sequence of minimum points of $\{f_n^k\}_{k,n}$ and $x$ is the unique minimizer of the strongly convex function $f$, satisfying:
\begin{enumerate}
\renewcommand{\theenumi}{\Roman{enumi}}
\item\label{lemstconfun2i1} the sequence $\{f_n^k\}$ is pointwise convergent to $f$ as $n$ tends to infinity uniformly in $k\geqslant 0$,
\item\label{lemstconfun2i2} $\displaystyle\limsup_{n\rightarrow\infty} \sup_{k\geqslant 0}\big(f(x_n^k)-f_n^k(x_n^k)\big)\leqslant 0$,
\end{enumerate}
then $x_n^k$ converges to $x$ uniformly in $k\geqslant 0$ as $n\rightarrow \infty$.
\end{lemma}
\begin{theorem}
Let $\{x_n\}$ be an almost periodic sequence in Hadamard space $\mathscr H$. Then the sequence $\{x_n\}$ is almost convergent.
\end{theorem}
\begin{proof}
Since $\{x_n\}$ is almost periodic, by \cite[Proposition 3.3]{khatibzadehpouladi2019} for each $x$, $\{d(x_n,x)\}$ is almost periodic, also it is easy to check that for each $x$, $\{d^2(x_n,x)\}$ is almost periodic. By \cite{lorentz1948}(see also \cite{bohr1947almost}) the scalar sequence $\{d^2(x_n,x)\}$ is almost convergent for all $x\in \mathscr H$. Define:
\begin{equation}\label{eqmainresu1}
\mathcal{F}_n^k(x):=\displaystyle\frac{1}{n}\sum_{i=0}^{n-1}d^2(x_{k+i},x),
\end{equation}
and
\begin{equation}\label{eqmainresu2}
\mathcal{F}(x):=\displaystyle\lim_{n\rightarrow\infty}\frac{1}{n}\sum_{i=0}^{n-1}d^2(x_{k+i},x)\quad \text{uniformly in}\ k\geqslant 0.
\end{equation}
Almost convergence of $\big\{d^2(x_n,x)\big\}$ for any $x\in \mathscr H$ shows that \eqref{eqmainresu2} is well defined. By the strong convexity of $d^2(\cdot,x)$, the functions $\mathcal{F}_n^k$ and $\mathcal{F}$ are strongly convex and therefore have unique minimizers $\sigma_n^k$ and $\sigma$ respectively. By analogous argument of \cite[Theorem 4.4]{khatibzadehpouladi2019} using of Lemma \ref{lemstconfun2}, we conclude that $\sigma_n^k$ converges to $\sigma$ uniformly in $k\geqslant 0$ as $n\rightarrow \infty$ or the sequence $\{x_n\}$ is almost convergent and this completes the proof.
\end{proof}
Every $N$-periodic sequence is almost periodic and by the previous theorem, it is almost convergent. We prove that it almost converges to the mean of its $N$ first points. 
\begin{theorem}
Let $\{x_n\}$ be a $N-$periodic sequence in Hadamard space $\mathscr H$. Then the sequence $\{x_n\}$ is almost convergent to $\sigma_N$ defined in Definition \ref{Karcher mean}.
\end{theorem}
\begin{proof}
In Definition \ref{Karcher mean}, we see that $\sigma_n$ or the Karcher mean of $x_0,\cdots,x_{n-1}$ is the unique minimizer of the function \eqref{e22} and $\sigma_n^k$ or the karcher mean of $x_k,\cdots,x_{k+n-1}$ is the unique minimizer of the function \eqref{e23}. By part \eqref{l15i1} of Lemma \ref{l15}, we have:
\begin{equation*}
d^2\big( \sigma_n^k ,\sigma_N\big)\leqslant \frac{1}{n}\displaystyle\sum_{i=0}^{n-1} d^2(x_{k+i},\sigma_N)-\frac{1}{n}\displaystyle\sum_{i=0}^{n-1} d^2(x_{k+i},\sigma_n^k).
\end{equation*}
Without loss of generality, for all $n$ we can suppose that $n=tN+r$, $0\leq r<N$. Now by $N-$periodicity of the sequence $\{x_n\}$ in step $\star$ and the definition of $\sigma_N$ in step $\star\star$, we obtain:
\allowdisplaybreaks\begin{eqnarray}
d^2\big( \sigma_n^k ,\sigma_N\big) &\leqslant & \frac{1}{n}\displaystyle\sum_{i=0}^{tN-1} d^2(x_{k+i},\sigma_N) +\frac{1}{n}\displaystyle\sum_{i=tN}^{tN+r-1} d^2(x_{k+i},\sigma_N) \nonumber \\
&&-\frac{1}{n}\displaystyle\sum_{i=0}^{tN-1} d^2(x_{k+i},\sigma_n^k)- \frac{1}{n}\displaystyle\sum_{i=tN}^{tN+r-1} d^2(x_{k+i},\sigma_n^k) \nonumber \\
&=& \frac{t}{n}\displaystyle\sum_{i=0}^{N-1} d^2(x_{k+i},\sigma_N) +\frac{1}{n}\displaystyle\sum_{i=tN}^{tN+r-1} d^2(x_{k+i},\sigma_N) \nonumber \\
&&-\frac{t}{n}\displaystyle\sum_{i=0}^{N-1} d^2(x_{k+i},\sigma_n^k)- \frac{1}{n}\displaystyle\sum_{i=tN}^{tN+r-1} d^2(x_{k+i},\sigma_n^k) \nonumber \\
&\leqslant &\frac{t}{n}\displaystyle\sum_{i=k}^{k+N-1} d^2(x_i,\sigma_N)-\frac{t}{n}\displaystyle\sum_{i=k}^{k+N-1} d^2(x_i,\sigma_n^k)+\frac{1}{n}\displaystyle\sum_{i=tN}^{tN+r-1} d^2(x_{k+i},\sigma_N) \nonumber\\
&=&\frac{t}{n}\Big(\displaystyle\sum_{i=0}^{N-1} d^2(x_i,\sigma_N)+\displaystyle\sum_{i=N}^{k+N-1} d^2(x_i,\sigma_N)-\displaystyle\sum_{i=0}^{k-1} d^2(x_i,\sigma_N) \Big)\nonumber\\
&&+\frac{t}{n}\Big(-\displaystyle\sum_{i=0}^{N-1} d^2(x_i,\sigma_n^k)+\displaystyle\sum_{i=0}^{k-1} d^2(x_i,\sigma_n^k)-\displaystyle\sum_{i=N}^{k+N-1} d^2(x_i,\sigma_n^k)\Big)\nonumber\\
&&+\frac{1}{n}\displaystyle\sum_{i=tN}^{tN+r-1} d^2(x_{k+i},\sigma_N) \nonumber \\
&\overset{\star}{=}&\frac{t}{n}\displaystyle\sum_{i=0}^{N-1} d^2(x_i,\sigma_N)-\frac{t}{n}\displaystyle\sum_{i=0}^{N-1} d^2(x_i,\sigma_n^k)+\frac{1}{n}\displaystyle\sum_{i=tN}^{tN+r-1} d^2(x_{k+i},\sigma_N) \nonumber \\
&\overset{\star\star}{\leqslant}&\frac{t}{n}\displaystyle\sum_{i=0}^{N-1} d^2(x_i,\sigma_n^k)-\frac{t}{n}\displaystyle\sum_{i=0}^{N-1} d^2(x_i,\sigma_n^k)+\frac{1}{n}\displaystyle\sum_{i=tN}^{tN+r-1} d^2(x_{k+i},\sigma_N) \nonumber \\
&\leqslant &\displaystyle\frac{r}{n}\sup_{tN\leqslant i\leqslant tN+r-1} d^2(x_{k+i},\sigma_N),
\end{eqnarray} 
hence
\begin{equation*}
\displaystyle\sup_{k\geqslant 0} d^2\big( \sigma_n^k ,\sigma_N\big) \leqslant \displaystyle\frac{r}{n}\sup_{k\geqslant 0}\sup_{tN\leqslant i\leqslant tN+r-1} d^2(x_{k+i},\sigma_N).
\end{equation*}
Now letting $n\rightarrow\infty$ completes the proof.
\end{proof}
\bibliographystyle{amsplain}  

\bibliography{shortbib}
\end{document}